\documentclass{article}
\usepackage{amsfonts}
\usepackage{amsmath}
\usepackage{amssymb}

\newtheorem{theorem}{Theorem}[section]

\newtheorem{definition}[theorem]{Definition}

\newtheorem{example}[theorem]{Example}
\newtheorem{lemma}[theorem]{Lemma}
\newtheorem{proposition}[theorem]{Proposition}
\newenvironment{proof}{\par\noindent{\bf Proof \,}}{$\hfill \Box$\par\bigskip}

\input{tcilatex}
\begin{document}

\title{On Modules over a G--set}
\author{Mehmet Uc, Mustafa Alkan}
\maketitle

\begin{abstract}
Let $R$ be a commutative ring with unity, $M$ a module over $R$ and let $S$
be a $G$--set for a finite group $G$. We define a set $MS$ to be the set of
elements expressed as the formal finite sum of the form $\sum\limits_{s\in
S}m_{s}s$ where $m_{s}\in M$. The set $MS$ is a module over the group ring $%
RG$ under the addition and the scalar multiplication similar to the $RG$%
--module $MG$ defined by Kosan, Lee and Zhou in \cite{Kosan}. With this
notion, we not only generalize but also unify the theories of both of the
group algebra and the group module, and we also establish some significant
properties of $(MS)_{RG}$. In particular, we describe a method for
decomposing a given $RG$--module $MS$ as a direct sum of $RG$--submodules.
Furthermore, we prove the semisimplicity problem of $(MS)_{RG}$ with regard
to the properties of $M_{R}$, $S$ and $G$.
\end{abstract}

\section{Introduction}

Throughout this paper, $G$ is a finite group with identity element $e$, $R$
is a commutative ring with unity $1$, $M$ is an $R$--module, $RG$ is the
group ring, $H\leq G$ denotes that $H$ is a subgroup of $G$ and $S$ is a $G$%
--set with a group action of $G$ on $S$. If $N$ is an $R$--submodule of $M$,
it is denoted by $N_{R}\leq M_{R}$.

$MS$ denote the set of all formal expression of the form $\dsum\limits_{s\in
S}m_{s}s$ where $m_{s}\in M$ and $m_{s}=0$ for almost every $s$. For
elements $\mu =\dsum\limits_{s\in S}m_{s}s$, $\eta =\dsum\limits_{s\in
S}n_{s}s\in MS$, by writing $\mu =\eta $ we mean $m_{s}=n_{s}$ for all $s\in
S$.

We define the sum in $MS$ componentwise%
\begin{equation*}
\mu +\eta =\dsum\limits_{s\in S}(m_{s}+n_{s})s
\end{equation*}

It is clear that $MS$ is an $R$--module with the sum defined above and the
scalar product of $\dsum\limits_{s\in S}m_{s}s$ by $r\in R$ that is $%
\dsum\limits_{s\in S}(rm_{s})s$.

For $\rho =\dsum\limits_{g\in G}r_{g}g\in RG$, the scalar product of $%
\dsum\limits_{s\in S}m_{s}s$ by $\rho $ is%
\begin{eqnarray*}
\rho \mu &=&\dsum\limits_{s\in S}r_{g}m_{s}(sg),\text{ }sg=s^{\prime }\in S,
\\
&=&\dsum\limits_{s^{\prime }\in S}m_{s^{\prime }}s^{\prime }\in MS
\end{eqnarray*}

It is easy to check that $MS$ is a left module over $RG$, and also as an $R$%
--module, it is denoted by $(MS)_{RG}$ and $(MS)_{R}$, respectively. The $RG$%
--module $MS$ is called $G$--set module of $S$ by $M$ over $RG$. It is clear
that $MS$ is also a $G$--set. If $S$ is a $G$--set and $H$ is a subgroup of $%
G$, then $S$ is also an $H$--set and $MS$ is an $RH$--module. In addition,
if $S$ is a $G$--set and a group, and $M=R$, then it is easy to verify that $%
RS$ is a group algebra. On the other hand, if a group acts on itself by
multiplication then naturally we have $(MS)_{RG}=(MG)_{RG}$. Since there is
a bijective correspondence between the set of actions of $G$ on a set $S$
and the set of homomorphisms from $G$ to $\Sigma _{S}$ ($\Sigma _{S}$ is the
group of permutations on $S$), the $G$--set modules is a large class of $RG$%
--modules and we would say that $(MG)_{RG}$ introduced in \cite{Kosan}
considering the group acting itself by multiplication is a first example of
the $G$--set modules. That is why the notion of the $RG$--module $MS$
presents a generalization of the structure and discussions of $RG$--module $%
MG$ and some principal module-theoretic questions arise out of the structure
of $(MS)_{RG}$. Therefore, this new concept generalizes not only the group
algebra but also the group module, and also unifies the theory of these two
concepts.

The purpose of this paper is to introduce the concept of the $RG$--module $%
MS $, and show the close connection between the properties of $(MS)_{RG}$, $%
M_{R}$, $S$ \ and $G$. The semisimplicity of $(MS)_{RG}$ with regard to the
properties of $M_{R}$, $S$ and $G$ and the decomposition of $(MS)_{RG}$ into 
$RG$--submodules will occupy a significant portion of this paper. In Section
1, we present some examples and some properties of $(MS)_{RG}$ to show that
an $R$--module can be extended to $RG$--modules in various ways via the
change of the $G$--set and the group ring. In Section 2, we give our first
major result about the decomposition of a given $RG$--module $MS$ as a
direct sum of $RG$--submodules. In Section 3, in order to go further into
the structure of $(MS)_{RG}$, we first require $\varepsilon _{MS}$ that is
an extension of the usual augmentation map $\varepsilon _{R}$ and the kernel
of $\varepsilon _{MS}$ denoted by $\triangle _{G}(MS)$. Then we give the
condition for when $\triangle _{G}(MS)$ is an $RG$--submodule of $(MS)_{RG}$%
. Finally, we are interested in the semisimplicity of $(MS)_{RG}$ according
to the properties of $M_{R}$, $S$ and $G$.

We start to set out the idea of $G$--set modules in more detail by
considering some examples of $G$--set modules and establishing some
properties of $(MS)_{RG}$. The following examples for $(MS)_{RG}$ show how
useful the notion of $G$--set module for extension of an $R$--module $M$ to
an $RG$--module. They also point the relations among $G$--set $S$, $RG$%
--module $MS$, $G$ and $H$ where $H\leq G$. Example \ref{ex1} shows that for
different group actions on different $G$--sets of the same finite group we
get different extensions of an $R$--module $M$ to an $RG$--module. Moreover,
we see that these are also $RH$--modules unsurprisingly in Example \ref{ex2}.

\begin{example}
\label{ex1}Let $M$ be an $R$--module, $G=D_{6}=\left\langle
a,b:a^{3}=b^{2}=e,b^{-1}ab=a^{-1}\right\rangle $ and $r=\dsum\limits_{g\in
D_{6}}r_{g}g=r_{1}e+r_{2}a+r_{3}a^{2}+r_{4}b+r_{5}ba+r_{6}ba^{2}\in RD_{6}$.

\begin{enumerate}
\item Let $S=G$ and let the group act itself by multiplication. Then $MS=MG$
is an $RG$--module.

\item Let $S=\left\{ D_{6},C_{3},C_{2},Id\right\} $ and let $G$ act on its
set of subgroups $C_{3}=\left\langle a:a^{3}=e\right\rangle \leq D_{6}$, $%
C_{2}=\left\langle b:b^{2}=e\right\rangle \leq D_{6}$, $Id=\left\{ e\right\}
\leq D_{6}$ by $g\ast H=gHg^{-1}$ for $H\leq G$, $g\in G$. Then $%
MS=\{\dsum\limits_{s\in
S}m_{s}s=m_{Id}Id+m_{C_{2}}C_{2}+m_{C_{3}}C_{3}+m_{D_{6}}D_{6}\mid m_{s}\in
M\}$ and we get%
\begin{eqnarray*}
r\mu &=&\left(
r_{1}m_{1}+r_{2}m_{1}+r_{3}m_{1}+r_{4}m_{1}+r_{5}m_{1}+r_{6}m_{1}\right) Id
\\
&&+\left(
r_{1}m_{C_{2}}+r_{2}m_{C_{2}}+r_{3}m_{C_{2}}+r_{4}m_{C_{2}}+r_{5}m_{C_{2}}+r_{6}m_{C_{2}}\right) C_{2}
\\
&&+\left(
r_{1}m_{C_{3}}+r_{2}m_{C_{3}}+r_{3}m_{C_{3}}+r_{4}m_{C_{3}}+r_{5}m_{C_{3}}+r_{6}m_{C_{3}}\right) C_{3}
\\
&&+\left(
r_{1}m_{D_{6}}+r_{2}m_{D_{6}}+r_{3}m_{D_{6}}+r_{4}m_{D_{6}}+r_{5}m_{D_{6}}+r_{6}m_{D_{6}}\right) D_{6}%
\text{.}
\end{eqnarray*}

\item Let $S=\left\{ K_{1}=\left\{ e,b\right\} ,K_{2}=\left\{ a,ba\right\}
,K_{3}=\left\{ a^{2},ba^{2}\right\} \right\} $ that is the set of right
cosets of a fixed subgroup $H=C_{2}=\left\langle b:b^{2}=e\right\rangle \leq
D_{6}$ and let $G$ act on $S$ by $g\ast (Hx)=H(gx)$ for $x,g\in G$. Then $%
MS=\{\dsum\limits_{s\in
S}m_{s}s=m_{K_{1}}K_{1}+m_{K_{2}}K_{2}+m_{K_{3}}K_{3}\mid m_{s}\in M\}$ and
we have the following relations such that 
\begin{equation*}
\begin{array}{ccc}
K_{1}1=K_{1} & K_{2}1=K_{2} & K_{3}1=K_{3} \\ 
K_{1}a=K_{2} & K_{2}a=K_{1} & K_{3}a=K_{1} \\ 
K_{1}a^{2}=K_{3} & K_{2}a^{2}=K_{3} & K_{3}a^{2}=K_{2} \\ 
K_{1}b=K_{1} & K_{2}b=K_{3} & K_{3}b=K_{2} \\ 
K_{1}ba=K_{2} & K_{2}ba=K_{1} & K_{3}ba=K_{3} \\ 
K_{1}ba^{2}=K_{3} & K_{2}ba^{2}=K_{2} & K_{3}ba^{2}=K_{1}.%
\end{array}%
\end{equation*}%
So, we get%
\begin{eqnarray*}
r\mu &=&\left(
r_{1}m_{K_{1}}+r_{4}m_{K_{1}}+r_{3}m_{K_{2}}+r_{5}m_{K_{2}}+r_{2}m_{K_{3}}+r_{6}m_{K_{3}}\right) K_{1}
\\
&&+\left(
r_{2}m_{K_{1}}+r_{5}m_{K_{1}}+r_{1}m_{K_{2}}+r_{6}m_{K_{2}}+r_{3}m_{K_{3}}+r_{4}m_{K_{3}}\right) K_{2}
\\
&&+\left(
r_{3}m_{K_{1}}+r_{6}m_{K_{1}}+r_{2}m_{K_{2}}+r_{4}m_{K_{2}}+r_{1}m_{K_{3}}+r_{5}m_{K_{3}}\right) K_{3}.
\end{eqnarray*}
\end{enumerate}
\end{example}

\begin{example}
\label{ex2}Let $M$ be an $R$--module, $G=D_{6}=\left\langle
a,b:a^{3}=b^{2}=e,b^{-1}ab=a^{-1}\right\rangle $, $H=C_{3}=\left\langle
a:a^{3}=e\right\rangle \leq D_{6}$ and $k=\dsum\limits_{g\in
D_{6}}k_{g}g=k_{1}e+k_{2}a+k_{3}a^{2}\in RC_{3}$.

\begin{enumerate}
\item Let $S=G$ and let the group act itself by multiplication. Then $MS=MG$
is an $RH$--module.

\item Let $S=\left\{ D_{6},C_{3},C_{2},Id\right\} $ with the group action
defined in Example 1.1 (2). For $\mu =\dsum\limits_{s\in
S}m_{s}s=m_{Id}Id+m_{C_{2}}C_{2}+m_{C_{3}}C_{3}+m_{D_{6}}D_{6}\in MS$, we
get 
\begin{eqnarray*}
k\mu &=&\left( k_{1}m_{1}+k_{2}m_{1}+k_{3}m_{1}\right) Id+\left(
k_{1}m_{C_{2}}+k_{2}m_{C_{2}}+k_{3}m_{C_{2}}\right) C_{2} \\
&&+\left( k_{1}m_{C_{3}}+k_{2}m_{C_{3}}+k_{3}m_{C_{3}}\right) C_{3}+\left(
k_{1}m_{D_{6}}+k_{2}m_{D_{6}}+k_{3}m_{D_{6}}\right) D_{6}.
\end{eqnarray*}

\item Let $S=\left\{ K_{1}=\left\{ e,b\right\} ,K_{2}=\left\{ a,ba\right\}
,K_{3}=\left\{ a^{2},ba^{2}\right\} \right\} $ with the group action defined
in Example 1.1 (3). For $\mu =\dsum\limits_{s\in
S}m_{s}s=m_{K_{1}}K_{1}+m_{K_{2}}K_{2}+m_{K_{3}}K_{3}\in MS$, we get%
\begin{eqnarray*}
k\mu &=&\left( k_{1}m_{K_{1}}+k_{3}m_{K_{2}}+k_{2}m_{K_{3}}\right)
K_{1}+\left( k_{2}m_{K_{1}}+k_{1}m_{K_{2}}+k_{3}m_{K_{3}}\right) K_{2} \\
&&+\left( k_{3}m_{K_{1}}+k_{2}m_{K_{2}}+k_{1}m_{K_{3}}\right) K_{3}
\end{eqnarray*}
\end{enumerate}
\end{example}

Now, we make a point of some relations between the $R$--submodules of $M$
and the $RG$--submodules of $MS$ by the following results.

\begin{lemma}
Let $N_{1}$, $N_{2}$ be $R$--submodules of $M$. Then $N_{1}S+N_{2}S=MS$ if
and only if $N_{1}+N_{2}=M$.
\end{lemma}

\begin{proof}
Let $N_{1}S+N_{2}S=NS$. Take $m\in M$ and so $ms\in MS$ for any $s\in S$. We
write $ms=\dsum\limits_{s_{i}\in S}n_{s_{i}}s_{i}+\dsum\limits_{s_{j}\in
S}n_{s_{j}}s_{j}$ for $\dsum\limits_{s_{i}\in S}n_{s_{i}}s_{i}\in N_{1}S$
and $\dsum\limits_{s_{j}\in S}n_{s_{j}}s_{j}\in N_{2}S$ where $n_{s_{i}}\in
N_{1}$, $n_{s_{j}}\in N_{2}S$. So, there exists $i,j$ such that $%
m=m_{s_{i}}+m_{s_{j}}$.

Let $N_{1}+N_{2}=M$ and $\mu =\dsum\limits_{s\in S}m_{s}s\in MS$. For all $%
s\in S$, we can write $m_{s}=n_{s}+n_{s}^{\prime }$ where $n_{s}\in N_{1}$, $%
n_{s}^{\prime }\in N_{2}$. Hence, $\mu =\dsum\limits_{s\in
S}n_{s}s+\dsum\limits_{s\in S}n_{s}^{\prime }s$, and so $N_{1}S+N_{2}S=NS$.
\end{proof}

\begin{lemma}
Let $N_{1}$, $N_{2}$ be $R$--submodules of $M$. Then $N_{1}S\cap N_{2}S=0$
if and only if $N_{1}\cap N_{2}=0$.
\end{lemma}

\begin{proof}
Let $N_{1}S+N_{2}S=0$. Take $n\in N_{1}\cap N_{2}$, and so $ns\in N_{1}S\cap
N_{2}S$. So, $n=0$ since $ns=0$.

Conversely, let $N_{1}\cap N_{2}=0$. Take $\eta =\dsum\limits_{s\in
S}n_{s}s\in N_{1}S\cap N_{2}S$. So $n_{s}\in N_{1}\cap N_{2}$ and $n_{s}=0$
for all $s\in S$. Hence, $N_{1}S\cap N_{2}S=0$.
\end{proof}

From \cite{Alp} we recall that if $G$ is a finite group, $S$ and $T$ are $G$%
--sets, then $\varphi :S\longrightarrow T$ is said to be a $G$--set
homomorphism if $\varphi (gs)=g\varphi (s)$ for any $g\in G$, $s\in S$. If $%
\varphi $ is bijective, then $\varphi $ is a $G$--set isomorphism. Then we
say that $S$ and $T$ are isomorphic $G$--sets, and we write $S\simeq T$.

For $s\in S$, $Gs=\left\{ gs:g\in G\right\} $ is the orbit of $s$. It is
easy to see that $Gs$ is also a $G$--set under the action induced from that
on $S$. In addition, a subset $S^{\prime }$ of $S$ is a $G$--set under the
action induced from $S$ if and only if $S^{\prime }$ is a union of orbits.

\begin{proposition}
Let $M$ be an $R$--module, $N$ an $R$--submodule of $M$, $G$ a finite group, 
$S$ a $G$--set. Then $\frac{MS}{NS}\simeq (\frac{M}{N})S$.
\end{proposition}

\begin{proof}
We know that $NS$ is an $RG$--submodule of $MS$. Define a map $\theta $ such
that%
\begin{equation*}
\begin{array}{cccc}
\theta : & MS & \longrightarrow & (\frac{M}{N})S%
\end{array}%
\text{, }%
\begin{array}{cccc}
\mu =\dsum\limits_{s\in S}m_{s}s & \longmapsto & \theta (\mu ) & 
=\dsum\limits_{s\in S}(m_{s}+N)s%
\end{array}%
\end{equation*}%
\begin{eqnarray*}
\theta (g\mu ) &=&\theta (g\dsum\limits_{s\in S}m_{s}s) \\
&=&g\theta (\mu )
\end{eqnarray*}%
So, $\theta $ is a $G$--set homomorphism. It is clear that $\theta $ is a $G$%
--set epimomorphism. Furthermore, $\theta $ is an $RG$--epimorphism and we
get $\ker \theta =NS$.
\end{proof}

\begin{lemma}
\label{mingset}Any proper subset of an orbit $Gs$ of $s\in S$ is not a $G$%
--set under the action induced from $S$.
\end{lemma}

\begin{proof}
Suppose that a proper subset $T$ of an orbit $Gs$ of $s\in S$ is a $G$--set.
Then there exist $sg\in G$, $sg\notin T$ for some $g\in G$. Take an element $%
sh$ in $T$, $h\in G$, and so%
\begin{eqnarray*}
(gh^{-1})(hs) &=&g(h^{-1}(hs)) \\
&=&gs\notin T.
\end{eqnarray*}%
Hence, we call the orbit $Gs$ of $s\in S$ the minimal\ $G$--set. Moreover, $%
S=\bigcup\limits_{i\in I}Gs_{i}$ where $I$ denotes the index of disjoint
orbits of $S$. Hence, we have 
\begin{equation*}
MS=M(\bigcup\limits_{i\in I}Gs_{i}).
\end{equation*}
\end{proof}

\begin{lemma}
Let $N$ be an $R$--submodule of an $R$--module $M$, $S$ a $G$--set. Let $I$
denote the index of disjoint orbits of $S$, $J$ a subset of $I$ and $%
S^{\prime }=\bigcup\limits_{j\in J}Gs_{j}$ and let $Gs_{i}$ be an orbit $Gs$
of $s_{i}\in S$ for $i\in I$. Then we have the following results:
\end{lemma}

\begin{enumerate}
\item $NGs_{i}$ is an $RG$--submodule of $MS$ for $s_{i}\in S$. Moreover, $%
NGs_{i}$ is a minimal $RG$--submodule of $MS$ containg $N$ under the action
induced from that on $S$.

\item $NS^{\prime }=N(\bigcup\limits_{j\in J}Gs_{j})=\bigcup\limits_{j\in
J}(NGs_{j})$.

\item $NS^{\prime }$\ is an $RG$--submodule of $MS$.
\end{enumerate}

\begin{proof}
\begin{enumerate}
\item It is clear that $NGs_{i}\subseteq MS$. Let $\eta =\sum\limits_{g\in
G}n_{g}gs_{i}\in NGs_{i}$ , $r\in R$, $h\in G$. Then we have $r\eta \in
NGs_{i}$ and $h\eta =h(\sum\limits_{g\in G}n_{g}gs_{i})=\sum\limits_{g\in
G}n_{g}hgs_{i}=\sum\limits_{hg=g^{\prime }\in G}n_{g}g^{\prime }s_{i}\in
NGs_{i}$. Hence, $NGs_{i}$ is an $RG$--submodule of $MS$. Assume that there
is an $RG$--submodule $N_{1}$ of $MS$ such that $N_{R}\leq (N_{1})_{RG}\leq
(NGs_{i})_{RG}$. Take an element $n\in N$, and so $nhs_{i}\in N_{1}$ for
some $h\in G$ since $(N_{1})_{RG}\leq (NGs_{i})_{RG}$. Then $%
h^{-1}(nhs_{i})=(nes_{i})=ns_{i}\in N_{1}$ and $g(ns_{i})=ngs_{i}\in N_{1}$
for all $g\in G$. This means that $N_{1}=NGs_{i}$.

\item[2,3.] Clear by the definition of $MS$.
\end{enumerate}
\end{proof}

\begin{lemma}
\label{modset}Let $L$ be an $RG$--submodule of $MS$, a fixed $s\in S$. Then,

\begin{enumerate}
\item $L_{s}=\{x\in M\mid $ there is $y\in L$ such that $y=xs+k$, $k\in MS\}$
is an $R$--submodule of $M$.

\item $S_{L}=\{s\in S\mid $there is $x\in M$, and also $k\in L$ such that $%
y=xs+k\in L$ $\}$ is a $G$--set in $S$ under the action induced from that on 
$S$.
\end{enumerate}
\end{lemma}

\begin{proof}
\begin{enumerate}
\item It is obvious that $L_{s}$ is in $M$. Let $x_{1}$, $x_{2}\in
L_{s^{\prime }}$ and $r\in R$. Then, there is $y_{1}=x_{1}s+k_{1}$, $%
y_{2}=x_{2}s+k_{2}\in L$ and $y_{1}+y_{2}=(x_{1}+x_{2})s+k_{1}+k_{2}\in L$
where $x_{1}+x_{2}\in MS$. Furthermore, $ry_{1}=rx_{1}s+rk_{1}\in L$, and so 
$rx_{1}\in L_{s}$.

\item Let $s\in S^{\prime }$ and $g$, $h\in G$. Then $\exists x\in M,\exists
k\in L$ such that $y=xs+k\in L$ and 
\begin{equation*}
xs+k=y=ey=e(xs+k)=xes+ek=xes+k
\end{equation*}%
So, $s=es$. Since $s$ is also an element of $S$, we have 
\begin{eqnarray*}
(hg)y &=&(hg)(xs+k) \\
&=&(hg)xs+(hg)k.
\end{eqnarray*}%
Hence, we get $(hg)s=h(gs)$.
\end{enumerate}
\end{proof}

\begin{lemma}
\label{simplems}Let $M$ be an $R$--module and $S$ a $G$--set. Let $I$ denote
the index of disjoint orbits of $S$ such that $S=\bigcup\limits_{i\in
I}Gs_{i}$ and let $Gs_{i}$ be an orbit of $s_{i}\in S$ for $i\in I$. If $%
NGs_{i}$ is a simple $RG$--submodule of $MS$, then $N$ is a simple $R$%
--submodule of $M$ and $G$ is a finite group whose order is invertible in $%
End_{R}(M)$ ($\left\vert G\right\vert ^{-1}\in End_{R}(M)$).
\end{lemma}

\begin{proof}
Assume that there is an $R$--submodule $L$ of $M$ such that $L\leq N\leq M$.
Then $(LGs_{i})_{RG}\leq (NGs_{i})_{RG}$, and by Lemma \ref{mingset} this is
a contradiction. So, $N$ is a simple $R$--submodule of $M$.
\end{proof}

\begin{theorem}
\label{assoc}Let $L$ be a simple $RG$--submodule of $MS$. Then there is a
unique simple $R$--submodule $N$ of $M$ and a unique orbit $Gs$ such that $%
L=NGs$.
\end{theorem}

\begin{proof}
For some $s\in S$, by Lemma \ref{modset} $L_{s}$ is a non-zero $R$--module.
And so, $L_{s}Gs\neq 0$ is an $RG$--submodule of $L$. Since $L$ is simple $%
RG $--submodule, we have $L_{s}Gs=L$. Then, by Lemma \ref{simplems} $L_{s}$
is a simple $R$--submodule of $M$.

Take an element $s^{\prime }\in S$ such that $L_{s^{\prime }}$ is non-zero $%
R $--submodule of $M$. Hence, $L_{s^{\prime }}Gs^{\prime }=L=L_{s}Gs$. Take
an element $x\in L_{s^{\prime }}Gs^{\prime }$. And so, we write 
\begin{equation*}
x=\sum\limits_{i=1}^{n}l_{i}g_{i}s^{\prime }=\sum\limits_{i=1}^{n}k_{i}g_{i}s
\end{equation*}%
where $l_{i}\in L_{s^{\prime }}$, $k_{i}\in L_{s}$, $g_{i}\in G$ and $%
n=\left\vert G\right\vert $. Then, there exists $g_{j}\in G$ such that $%
g_{1}s=g_{j}s^{\prime }$, and $s=g_{1}^{-1}g_{j}s^{\prime }$. So, we get $%
Gs=Gs^{\prime }$. That is why we can write%
\begin{equation*}
Gs=S_{L}=\{s\in S\mid \text{there is }x\in M,\text{\ and also }k\in L\text{
such that }y=xs+k\in L\}.
\end{equation*}%
Moreover, $N=L_{s}=L_{s^{\prime }}$ is unique by the definition of $MS$.
\end{proof}

On the other hand, the following example shows that the converse of the
theorem does not hold.

\begin{example}
Let $R=%
\mathbb{Z}
_{3}$, $M=%
\mathbb{Z}
_{3}$, $G=C_{2}=\left\langle a:a^{2}=e\right\rangle $ and $RG=%
\mathbb{Z}
_{3}C_{2}$. If $S=G$ and $G$ acts on itself by group multiplication then $MS=%
\mathbb{Z}
_{3}C_{2}$ where $%
\mathbb{Z}
_{3}C_{2}$ is semisimple $RG$--module since $\left\vert G\right\vert \leq
\infty $ and characteristic of $R$ does not divide $\left\vert G\right\vert $
by Maschke's Theorem. Since $%
\mathbb{Z}
_{3}C_{2}$ is semisimple there is a unique decomposition of $%
\mathbb{Z}
_{3}C_{2}$ by Artin-Weddernburn Theorem. Then, $%
\mathbb{Z}
_{3}C_{2}\simeq 
\mathbb{Z}
_{3}\oplus 
\mathbb{Z}
_{3}$ as $R$--module since $\left\vert C_{2}\right\vert =2$. Here, $%
\mathbb{Z}
_{3}$ is a simple $R$--submodule of $%
\mathbb{Z}
_{3}C_{2}$. Moreover, by \cite{Sehgal} we have $%
\mathbb{Z}
_{3}C_{2}\simeq 
\mathbb{Z}
_{3}C_{2}(\frac{1+a}{2})\oplus 
\mathbb{Z}
_{3}C_{2}(\frac{1-a}{2})$ as $RG$--module where $%
\mathbb{Z}
_{3}C_{2}(\frac{1+a}{2})$ and $%
\mathbb{Z}
_{3}C_{2}(\frac{1-a}{2})$ are simple $RG$--submodules of $%
\mathbb{Z}
_{3}C_{2}$. Let $N=%
\mathbb{Z}
_{3}$ that is a simple $R$--submodule of $M$. Hovewer, $NGs=%
\mathbb{Z}
_{3}C_{2}$ is not simple $RG$--module.
\end{example}

\begin{lemma}
\label{a5}Let $\{M_{i}:i\in I\}$ be a family of right $R-$modules, $G$ a
finite group and $S$ a $G-$set. Then%
\begin{equation*}
\left( \left( \dbigoplus\limits_{i\in I}M_{i}\right) S\right) _{RG}=\left(
\dbigoplus\limits_{i\in I}M_{i}S\right) _{RG}
\end{equation*}
\end{lemma}

\begin{proof}
Consider the following map 
\begin{equation*}
\begin{array}{ccc}
\left( \dbigoplus\limits_{i\in I}M_{i}\right) S & \longrightarrow & 
\dbigoplus\limits_{i\in I}M_{i}S%
\end{array}%
\text{, }%
\begin{array}{ccc}
\dsum\limits_{s\in S}(...,m_{s}^{(i)},...)s & \longmapsto & 
\dsum\limits_{s\in S}(...,m_{s}^{(i)}s,...)%
\end{array}%
\end{equation*}%
that is an isomorphism.
\end{proof}

\begin{theorem}
An $R$--module $M_{R}$ is projective if and only if $(MS)_{RG}$ is
projective.
\end{theorem}

\begin{proof}
Assume that $M_{R}$ is projective. Then for an index $I$, $(R)^{(I)}\simeq
M\oplus A$ where $A$ is a right $R$--module. So, by Lemma \ref{a5}%
\begin{eqnarray*}
((RS)^{(I)})_{RG} &\simeq &((R)^{(I)}S)_{RG} \\
&\simeq &((M\oplus A)S)_{RG} \\
&\simeq &(MS)_{RG}\oplus (AS)_{RG}
\end{eqnarray*}%
So, $(MS)_{RG}$ is projective.

Now, assume that $(MS)_{RG}$ is projective. Then $((RS)^{(I)})_{RG}\simeq
(MS)_{RG}\oplus B$ where $B$ is a right $RG$--module for some set $I$. All
this concerning modules are also $R$--modules and $((RS)^{(I)})_{R}\simeq
(MS)_{R}\oplus B_{R}$. $((RS)^{(I)})_{R}$ is a free module because $(RS)_{R}$
is free. Since $(MS)_{R}$ is direct summand of a free module, it is
projective. So, $M_{R}$ is projective.
\end{proof}

\section{The Decomposition of $(MS)_{RG}$}

The theme of this section is the examination of a $G$--set module $(MS)_{RG}$
through the study of a decomposition of it. The decompositions of $RG$ and $%
(MG)_{RG}$ obtained from the idempotent defined as $e_{H}=\frac{\hat{H}}{%
\left\vert H\right\vert }$ , where $\left\vert H\right\vert $ is the order
of $H$ and $\hat{H}=$ $\sum\limits_{h\in H}h$, explained in \cite{Sehgal}
and \cite{ucalkan}, respectively. A similar method give a criterion for the
decomposition of a $G$--set module $(MS)_{RG}$. In addition, $End_{RG}MS$
denotes all the $RG$--endomorphisms of $MS$.

\begin{lemma}
\label{idem}Let $M$ be an $R$-module and $H$ a normal subgroup of finite
group $G$. If $\left\vert H\right\vert $, the order of $H$,\ is invertible
in $R$ then $\widetilde{e}_{H}=\frac{\hat{H}}{\left\vert H\right\vert }\ $is
an idempotent in $End_{RG}(MS)$. Moreover, $\widetilde{e}_{H}$ is central in 
$End_{RG}(MS)$.
\end{lemma}

\begin{proof}
Firstly, we will show that $\widetilde{e}_{H}$ is an $RG$--homomorphism. We
start with proving that $\hat{H}g=g\hat{H}$for $g\in G$. Since for all $%
h_{i}\in H,$ there is $h_{ig}\in H$ such that $h_{i}g=gh_{ig},$ we have that 
$\hat{H}g=\sum\limits_{h_{i}\in H}h_{i}g=\sum\limits_{h_{i}\in H}gh_{ig}=g%
\hat{H}$. Therefore, $\frac{\hat{H}}{\left\vert H\right\vert }rg=rg\frac{%
\hat{H}}{\left\vert H\right\vert }$ and we have $\widetilde{e}_{H}(rgm)=rg%
\widetilde{e}_{H}(m)$ for $m\in MS$, $r\in R$ and $g\in G$. It is also clear
that $\widetilde{e}_{H}(m+n)=\widetilde{e}_{H}(m)+\widetilde{e}_{H}(n)$ for $%
m,n\in MS$, $g\in G$.

Secondly, by using the fact that $\hat{H}.\hat{H}=\left\vert H\right\vert .%
\hat{H}$, we get 
\begin{eqnarray*}
\widetilde{e}_{H}(\widetilde{e}_{H}(m)) &=&\widetilde{e}_{H}(\frac{\hat{H}}{%
\left\vert H\right\vert }m) \\
&=&\widetilde{e}_{H}(m)
\end{eqnarray*}%
So, $\widetilde{e}_{H}$ is an idempotent.

Finally, we prove that $\widetilde{e}_{H}$ is a central idempotent in $%
End_{RG}(MS)$. We will show that $\widetilde{e}_{H}$ commutes with every
element of $End_{RG}(MS)$. Let $f$ be in $End_{RG}(MS)$ and so $\hat{H}%
f(m)=f(\hat{H}m)$ for $m\in MS$. Thus, we have 
\begin{eqnarray*}
\widetilde{e}_{H}f(m) &=&\frac{\hat{H}}{\left\vert H\right\vert }f(m)\ \  \\
&=&f(\frac{\hat{H}}{\left\vert H\right\vert }m)=f\widetilde{e}_{H}(m).
\end{eqnarray*}
\end{proof}

For $\mu =\sum\limits_{g\in G}m_{g}g\in MG$ and $s_{i}\in S$, we write 
\begin{eqnarray*}
\mu s_{i} &=&\sum\limits_{g\in G}m_{g}(gs_{i}) \\
&=&\sum\limits_{gs_{i}\in S}m_{gs_{i}}(gs_{i})\in MS
\end{eqnarray*}%
Then for $i\in I$ and $\alpha \in M(Gs_{i})$, we write $\alpha
=\sum\limits_{gs_{i}\in Gs_{i}}m_{gs_{i}}gs_{i}$. Moreover, we write $\beta
=\sum\limits_{i\in I}\sum\limits_{gs_{i}\in Gs_{i}}m_{gs_{i}}gs_{i}$ for $%
\beta =\sum\limits_{s\in S}m_{s}s\in MS$ since $MS=M(\bigcup\limits_{i\in
I}Gs_{i})$.

Let $H$ be a normal subgroup of $G$. It is well known that on $G/H$ we have
the group action $g(tH)=gtH$ for $g,t\in G$. Consider $g(\sum\limits_{s\in
S}m_{s}(sH))=(\sum\limits_{s\in S}m_{s}(gsH))$ for $m_{s}\in M$.

Let $S^{\prime }$ $\subset S$ be a $G/H$--set. Then $S^{\prime
}=\bigcup\limits_{j\in J}G/Hs_{j}^{\prime }$ where $J$ denotes the index of
disjoint orbits of $S^{\prime }$ and $MS^{\prime }=M(\bigcup\limits_{j\in
J}G/Hs_{j}^{\prime })$. Then for $\eta =\sum\limits_{s^{\prime }\in
S^{\prime }}m_{s^{\prime }}s^{\prime }\in MS$, we can write $\eta
=\sum\limits_{j\in J}\sum\limits_{s^{\prime }\in G/Hs_{j}^{\prime
}}m_{s^{\prime }}s^{\prime }$.

Hence, we have the following result.

\begin{lemma}
Let $M$ be an $R$--module, $G$ a finite group, $H$ a normal subgroup of $G$, 
$S$ a $G$--set and $S^{\prime }$ $\subset S$ a $G/H$--set. Then $MS^{\prime
} $ is an $RG$--module with action defined as $g\eta =g(\sum\limits_{j\in
J}\sum\limits_{s^{\prime }\in G/Hs_{j}^{\prime }}m_{s^{\prime }}s^{\prime
})= $ $g(\sum\limits_{j\in J}\sum\limits_{s^{\prime }\in G/Hs_{j}^{\prime
}}m_{s^{\prime }}(tHs_{j}^{\prime })=\sum\limits_{j\in
J}\sum\limits_{s^{\prime }\in G/Hs_{j}^{\prime }}m_{s^{\prime
}}(gtHs_{j}^{\prime })$ where

$\eta =\sum\limits_{j\in J}\sum\limits_{s^{\prime }\in G/Hs_{j}^{\prime
}}m_{s^{\prime }}s^{\prime }\in MS^{\prime }$ and $s^{\prime
}=tHs_{j}^{\prime }$ for $t\in G$.
\end{lemma}

\begin{theorem}
\label{a1}Let $H$ be a normal subgroup of $G$, $\left\vert H\right\vert $
invertible in $R$ and $\widetilde{e}_{H}$, defined above, then we have $MS=%
\widetilde{e}_{H}.MS\oplus (1-\widetilde{e}_{H}).MS$ and there exists a $G/H$%
--set $S^{\prime }$ $\subset S$ such that $\widetilde{e}_{H}.MS\simeq
MS^{\prime }$. More precisely, 
\begin{equation*}
\widetilde{e}_{H}.MS=\widetilde{e}_{H}\left( M(\bigcup\limits_{i\in
I}Gs_{i})\right) \simeq M(\bigcup\limits_{i\in I}\widetilde{e}_{H}Gs_{i})
\end{equation*}
\end{theorem}

\begin{proof}
Firstly, we know that $MG=\widetilde{e}_{H}.MG\oplus (1-\widetilde{e}%
_{H}).MG $ and $\widetilde{e}_{H}.MG\simeq M(G/H)$ by the theorem in \cite%
{ucalkan}. Since $\widetilde{e}_{H}$ is a central idempotent by Lemma \ref%
{idem}, we get $MS=\widetilde{e}_{H}.MS\oplus (1-\widetilde{e}_{H}).MS$.
Now, consider $\theta :G\longrightarrow G.\widetilde{e}_{H}$ where $g\mapsto
g\widetilde{e}_{H}$. This is a group homomorphism since $\theta (gh)=gh%
\widetilde{e}_{H}=gh\widetilde{e}_{H}^{2}=g\widetilde{e}_{H}h\widetilde{e}%
_{H}=\theta (g)\theta (h)$. It is clear that $\theta $ is a group
epimorphism. We have $ker\theta =\left\{ g\in G\mid g\widetilde{e}_{H}=%
\widetilde{e}_{H}\right\} =\left\{ g\in G\mid (g-1)\widetilde{e}%
_{H}=0\right\} =H$ since $(g-1)\frac{\hat{H}}{\left\vert H\right\vert }=0$
and $g\hat{H}=\hat{H}$ for $g\in H$. Moreover, we get $\frac{G}{er\theta }=%
\frac{G}{H}\simeq $ Im$\theta $ $=G\widetilde{e}_{H}$. So, 
\begin{equation*}
\widetilde{e}_{H}.MS=\widetilde{e}_{H}\left( M(\bigcup\limits_{i\in
I}Gs_{i})\right) =M(\bigcup\limits_{i\in I}G\widetilde{e}_{H}s_{i})\simeq
M(\bigcup\limits_{i\in I}(G/H)s_{i})
\end{equation*}%
Since $gHs_{i}=gHs_{l}$ for $s_{i},s_{l}\in S$, $i,l\in I$, we get a $G/H$%
--set $S^{\prime }$ $\subset S$ where $\bigcup\limits_{j\in
J}(G/H)s_{j}=S^{\prime }\subseteq S$. Hence 
\begin{equation*}
\widetilde{e}_{H}.MS\simeq M(\bigcup\limits_{i\in
I}(G/H)s_{i})=M(\bigcup\limits_{j\in J}(G/H)s_{j})=MS^{\prime }
\end{equation*}%
So, $\widetilde{e}_{H}.MS\simeq MS^{\prime }$.
\end{proof}

\begin{theorem}
Let $M$ be an $R$--module and $G$ a finite group. For a $G$--set $%
S=\bigcup\limits_{i\in I}Gs_{i}$ ($I$ denotes the index of disjoint orbits
of $S)$, $MS\simeq \bigoplus\limits_{i\in I}MG\backslash \ker \theta _{i}$
where $\theta _{i}:MG\longrightarrow MGs_{i}$ are $RG$--epimorphisms.
\end{theorem}

\begin{proof}
Since $MGs_{i}\cap MGs_{j}=\emptyset $ for $i\neq j\in I$ where $%
S=\bigcup\limits_{i\in I}Gs_{i}$ and $I$ denotes the index of disjoint
orbits of $S$ , we have $MS=M(\bigcup\limits_{i\in
I}Gs_{i})=\bigoplus\limits_{i\in I}MGs_{i}$.

Consider 
\begin{equation*}
\begin{array}{cccc}
\theta _{i}: & MG & \longrightarrow & MGs_{i}%
\end{array}%
\text{, }%
\begin{array}{ccc}
\sum\limits_{g\in G}m_{g}g & \longmapsto & \sum\limits_{g\in G}m_{g}gs_{i}%
\end{array}%
\end{equation*}%
For $\mu =\sum\limits_{g\in G}m_{g}g\in MG$, $r\in R$, $h\in G$, we have 
\begin{eqnarray*}
\theta _{i}(r\mu ) &=&\theta _{i}(r\sum\limits_{g\in G}m_{g}g)=\theta
_{i}(\sum\limits_{g\in G}rm_{g}g)=\sum\limits_{g\in G}rm_{g}gs_{i} \\
&=&r\sum\limits_{g\in G}m_{g}gs_{i}=r\theta _{i}(\sum\limits_{g\in
G}m_{g}g)=r\theta _{i}(\mu )\text{.}
\end{eqnarray*}%
\begin{eqnarray*}
\theta _{i}(h\mu ) &=&\theta _{i}(h\sum\limits_{g\in G}m_{g}g)=\theta
_{i}(\sum\limits_{g\in G}m_{g}hg)=\sum\limits_{g\in G}m_{g}hgs_{i} \\
&=&h(\sum\limits_{g\in G}m_{g}gs_{i})=h\theta _{i}(\sum\limits_{g\in
G}m_{g}g)=h\theta _{i}(\mu ).
\end{eqnarray*}%
Hence, $\theta _{i}$ is an $RG$--homomorphism. It is clear that $\theta _{i}$
is an epimorphism. Moreover, $MG\backslash \ker \theta _{i}\simeq \func{Im}%
\theta _{i}=MGs_{i}$. Then,%
\begin{equation*}
MS=M(\bigcup\limits_{i\in I}Gs_{i})=\bigoplus\limits_{i\in I}MGs_{i}\simeq
\bigoplus\limits_{i\in I}MG\backslash \ker \theta _{i}\text{.}
\end{equation*}
\end{proof}

\section{Augmentation Map on MS and Semisimple $G$--set Modules}

In the theory of the group ring, the augmentation ideal denoted by $%
\triangle (RG)$ is the kernel of the usual augmentation map $\varepsilon
_{R} $ such that 
\begin{equation*}
\begin{array}{cccc}
\varepsilon _{R}: & RG & \longrightarrow & R%
\end{array}%
\text{, }%
\begin{array}{ccc}
\sum\limits_{g\in G}r_{g}g & \longmapsto & \sum\limits_{g\in G}r_{g}%
\end{array}%
.
\end{equation*}%
The augmentation ideal is always the nontrivial two-sided ideal of the group
ring and we have $\triangle (RG)=\left\{ \sum\limits_{g\in
G}r_{g}(g-1):r_{g}\in R,g\in G\right\} $. The augmentation ideal $\triangle
(RG)$ is of use for studying not only the relationship between the subgroups
of $G$ and the ideals of $RG$ but also the decomposition of $RG$ as direct
sum of subrings.

In \cite{Kosan}, $\varepsilon _{R}$ is extended to the following
homomorphism of $R$--modules 
\begin{equation*}
\begin{array}{cccc}
\varepsilon _{M}: & MG & \longrightarrow & M%
\end{array}%
\text{,\ }%
\begin{array}{ccc}
\sum\limits_{g\in G}m_{g}g & \longmapsto & \sum\limits_{g\in G}m_{g}%
\end{array}%
.
\end{equation*}%
The kernel of $\varepsilon _{M}$ is denoted by $\triangle (MG)$ and 
\begin{equation*}
\triangle (MG)=\left\{ \sum\limits_{g\in G}m_{g}(g-1):m_{g}\in M,g\in
G\right\} .
\end{equation*}
We devote this section to $\varepsilon _{MS}$ that is an extension of $%
\varepsilon _{M}$, and to the kernel of $\varepsilon _{MS}$ denoted by $%
\triangle _{G}(MS)$.

\begin{definition}
The map 
\begin{equation*}
\begin{array}{clcl}
\varepsilon _{MS}: & MS & \longrightarrow & M%
\end{array}%
\text{, }%
\begin{array}{lcl}
\dsum\limits_{s\in S}m_{s}s & \longmapsto & \dsum\limits_{s\in S}m_{s}%
\end{array}%
\end{equation*}%
is called augmentation map on $MS$\textbf{.\ }
\end{definition}

In addition, $\varepsilon _{MS}(m_{s}s_{1})=\varepsilon
_{MS}(m_{s}s_{2})=m_{s}$ for $m_{s}s_{1}$, $m_{s}s_{2}\in MS$ where $%
m_{s}\in M$, $s_{1},s_{2}\in S$, however $m_{s}s_{1}\neq m_{s}s_{2}$. Hence, 
$\varepsilon _{MS}$ is not one-to-one.

\begin{lemma}
\bigskip Let $M$ be an $R$--module, $G$ a group and $S$ a $G$--set. Then $%
\varepsilon _{MS}(r\mu )$ $=$ $\varepsilon (r)$ $\varepsilon _{MS}(\mu )$
for $\mu =\dsum\limits_{s\in S}m_{s}s\in MS$, $r=\dsum\limits_{g\in
G}r_{g}g\in RG$. In particular, $\varepsilon _{MS}$ is an $R$--homomorphism.
\end{lemma}

\begin{proof}
Let $\mu =\dsum\limits_{s\in S}m_{s}s\in MS$, $r=\dsum\limits_{g\in
G}r_{g}g\in RG$, then%
\begin{eqnarray*}
\varepsilon _{MS}(r\mu ) &=&\varepsilon _{MS}\left( \dsum\limits_{gs\in
S}(r_{g}m_{s})(gs)\right) \\
&=&\varepsilon _{MS}\left( \dsum\limits_{s^{\prime }\in S}m_{s^{\prime
}}s^{\prime }\right) ,\text{ }m_{s^{\prime }}=r_{g}m_{s},gs=s^{\prime }\in S,
\\
&=&\left( \dsum\limits_{g\in G}r_{g}\right) \left( \dsum\limits_{s\in
S}m_{s}\right) \\
&=&\varepsilon (r)\varepsilon _{MS}(\mu ).
\end{eqnarray*}%
In addition, for $\mu =\dsum\limits_{s\in S}m_{s}s$, $\eta
=\dsum\limits_{s\in S}n_{s}s\in MS$, $t\in R$, 
\begin{eqnarray*}
\varepsilon _{MS}(\mu +\eta ) &=&\varepsilon _{MS}(\dsum\limits_{s\in
S}\left( m_{s}+n_{s}\right) s) \\
&=&\dsum\limits_{s\in S}m_{s}+\dsum\limits_{s\in S}n_{s}
\end{eqnarray*}%
\begin{eqnarray*}
\varepsilon _{MS}(t\mu ) &=&\varepsilon _{MS}(\dsum\limits_{s\in S}\left(
tm_{s}\right) s) \\
&=&t\dsum\limits_{s\in S}m_{s}
\end{eqnarray*}
\end{proof}

Furhermore, 
\begin{equation*}
ker(\varepsilon _{MS})=\{\mu =\dsum\limits_{s\in S}m_{s}s\in MS\mid
\varepsilon _{MS}(\mu )=\varepsilon _{MS}(\dsum\limits_{s\in
S}m_{s}s)=\dsum\limits_{s\in S}m_{s}=0\}.
\end{equation*}%
It is clear that $ker(\varepsilon _{MS})\neq 0$ because for $%
m_{s}s_{1}+(-m_{s}s_{2})\in MS$, where $m\in M$, $s_{1}\neq s_{2}\in S$, we
have 
\begin{eqnarray*}
\varepsilon _{MS}(m_{s}s_{1}+(-m_{s}s_{2})) &=&\varepsilon
_{MS}(m_{s}s_{1})+\varepsilon _{MS}(-m_{s}s_{2}) \\
&=&0
\end{eqnarray*}%
Thus, $m_{s}s_{1}+(-m_{s}s_{2})\in er(\varepsilon _{MS})$. Moreover, we will
characterize the elements of the kernel of $\varepsilon _{MS}$ in detail.
For this purpose, we define $\triangle _{G,H}(MS)=\{\sum\limits_{h\in
H}(h-1)\mu _{h}\mid \mu _{h}\in MS\}$ where $H $ is a subgroup of finite
group $G$.

\begin{theorem}
Let $M$ be an $R$--module, $H$ a subgroup of $G$, $\left\vert H\right\vert $
invertible in $R$, $S$ a $G$--set and $\widetilde{e}_{H}$, defined in Lemma %
\ref{idem}. Then, $\triangle _{G,H}(MS)$ is an $RG$--module and $\triangle
_{G,H}(MS)=(1-\widetilde{e}_{H}).MS$.
\end{theorem}

\begin{proof}
$\triangle _{G,H}(MS)$ is obviously an $RG$--module. Now, take any element $%
\alpha \in \triangle _{G,H}(MS)$. Then we get 
\begin{eqnarray*}
\alpha &=&\sum\limits_{h\in H}(h-1)\mu _{h} \\
&=&\sum\limits_{h\in H}(h-1)(\sum\limits_{s\in S}m_{s}s) \\
&=&\sum\limits_{h\in H}(\sum\limits_{s\in S}m_{s}(h-1)s) \\
&=&\sum\limits_{h\in H}(\sum\limits_{s\in S}m_{s}(hs-s)) \\
&=&\sum\limits_{h\in H}(\sum\limits_{s\in S}m_{s}(hs-1)-(s-1)).
\end{eqnarray*}%
On the other hand, for any element $\beta \in (1-\widetilde{e}_{H}).MS$%
\begin{eqnarray*}
\beta &=&(1-\widetilde{e}_{H})\eta \\
&=&(1-\widetilde{e}_{H})(\sum\limits_{s\in S}n_{s}s) \\
&=&(1-\frac{\hat{H}}{\left\vert H\right\vert })(\sum\limits_{s\in S}n_{s}s)
\\
&=&-\frac{1}{\left\vert H\right\vert }(\sum\limits_{h\in
H}(h-1))(\sum\limits_{s\in S}n_{s}s) \\
&=&(\sum\limits_{h\in H}(h-1))(\sum\limits_{s\in S}n_{s}^{\prime }s) \\
&=&\sum\limits_{h\in H}(h-1)(\sum\limits_{s\in S}n_{s}^{\prime }s) \\
&=&\sum\limits_{h\in H}(\sum\limits_{s\in S}n_{s}^{\prime }(hs-1)-(s-1))
\end{eqnarray*}%
where $\eta \in MS$, $n_{s}^{\prime }=-\frac{1}{\left\vert H\right\vert }%
n_{s}$. Hence, $\beta \in \triangle _{G,H}(MS)$. Similarly, $\alpha \in
MS.(1-\widetilde{e}_{H})$.
\end{proof}

Furthermore, we write $\triangle _{G,G}(MS)=\triangle _{G}(MS)$. It is clear
that $ker(\varepsilon _{MS})=\triangle _{G}(MS)$ and we have $%
ker(\varepsilon _{MS})=\triangle _{G}(MS)=(1-\widetilde{e}_{G}).MS$.

Recall that $\triangle _{R}(G)$ is the augmetation ideal of $RG$ and for a
normal subgroup $N$ of $G$, $\triangle _{R}(G,N)$ denote the kernel of the
natural epimorphism $RG\longrightarrow R(G/N)$ induced by $G\longrightarrow
G/N$. Moreover, $\triangle _{R}(G,N)$ is a two-sided ideal of $RG$ generated
by $\triangle _{R}(N)$.

\begin{theorem}
If $N$ is a normal subgroup of $G$, then $\triangle _{G,N}(MS)=\triangle
_{R}(N).MS$.
\end{theorem}

\begin{proof}
We know that $\triangle _{R}(N)=\{\sum\limits_{n\in N}r_{n}(n-1)\mid
r_{n}\in R\}$ and $\triangle _{G,H}(MS)=\{\sum\limits_{h\in H}(h-1)\mu
_{h}\mid \mu _{h}\in MS\}$. For $\alpha =\sum\limits_{n\in N}r_{n}(n-1)\in
\triangle _{R}(N)$, $\mu =\sum\limits_{s\in S}m_{s}s\in MS$,%
\begin{eqnarray*}
\alpha \mu &=&\left( \sum\limits_{n\in N}r_{n}(n-1)\right) \left(
\sum\limits_{s\in S}m_{s}s\right) \\
&=&\sum\limits_{n\in N}r_{n}(n-1)\left( \sum\limits_{s\in S}m_{s}s\right) \\
&=&\sum\limits_{n\in N}(n-1)\left( \sum\limits_{s\in S}(r_{n}m_{s})s\right)
\\
&=&\sum\limits_{n\in N}(n-1)\mu _{n}
\end{eqnarray*}%
where $\mu _{n}=\sum\limits_{s\in S}(r_{n}m_{s})s\in MS$.
\end{proof}

In examination of the studies in group rings which make use of the theory of
group modules (see \cite{Con}, \cite{Kosan}, \cite{ucalkan}), the
semisimplicity problem of the $G$--set module arises. In \cite{Con}, the
generalized Maschke's Theorem states that a group ring $RG$ is a semisimple
Artinian ring if and only if $R$ is a semisimple Artinian ring, $G$ is
finite and $\left\vert G\right\vert ^{-1}\in R$. A module theoretic version
of the Maschke's Theorem is proven in \cite{Kosan}. This version states that
for a nonzero $R$--module $M$ and a group $G$, $MG$ is a semisimple module
over $RG$ if and only if $M$ is a semisimple module and $G$ is a finite
group whose order is invertible in $End_{R}(M)$ that is all the $R$%
--endomorphisms of $M$. The purpose of this section is generalizing the
Maschke's Theorem to the $G$--set modules to give the criterion for the
semisimplicity of a $G$--set module.

\begin{lemma}
\label{semlem1}Let $M$ be a nonzero $R$--module, $G$ a group, $S$ a $G$%
--set. If $X\cap \triangle _{G}(MS)=0$ for some nonzero $RG$--submodule $X$
of $(MS)_{RG}$, then $G$ is a finite group.
\end{lemma}

\begin{proof}
Firstly, we know that $\triangle _{G}(MS)$ is an $RG$--submodule of $%
(MS)_{RG}$. Assume that $G$ is an infinite group. Then for any $0\neq
x=m_{1}s_{1}+...+m_{k}s_{k}\in X$ where $s_{1},...,s_{k}\in S$ are distinct
and $m_{i}s_{i}\neq 0$, there is an element $g$ of $G$ such that $s_{i}g\neq
s_{j}$ for $1\leq i\leq k$. Hence, $(1-g)x=\sum\limits_{s_{i}\in
S}m_{i}s_{i}-\sum\limits_{s_{i}\in S}m_{i}gs_{i}\neq 0$, and also $(1-g)x\in
Y$ . On the other hand, $0\neq (1-g)x=\sum\limits_{s_{i}\in
S}m_{i}(s_{i}-1)-\sum\limits_{s_{i}\in S}m_{i}(gs_{i}-1)\in \triangle
_{G}(MS)$. Then, $X\cap \triangle _{G}(MS)\neq 0$ and this is a
contradiction.
\end{proof}

We recall the following lemma in \cite{Lam}, and also in \cite{Kosan}.

\begin{lemma}
\label{semlem2}\cite{Lam}\cite{Kosan}Let $X\leq Y$ be right $RG$--modules
and $G$ be a finite group whose order is invertible in $End_{R}(V)$. If $X$
is a direct summand of $Y$ as $R$--modules, then $X$ is a direct summand of $%
Y$ as $RG$--modules.
\end{lemma}

\begin{theorem}
Let $M$ be a nonzero $R$--module, $G$ a group, $S$ a $G$--set. Then, $MS$ is
a semisimple module over $RG$ if and only if $M$ is a semisimple $R$%
--module, $G$ is a finite group whose order is invertible in $End_{R}(M)$ ($%
\left\vert G\right\vert ^{-1}\in End_{R}(M)$).
\end{theorem}

\begin{proof}
Assume that $M$ is a semisimple $R$--module, $G$ is a finite group whose
order is invertible in $End_{R}(M)$. Let $Y$ be an $RG$--submodule of $MS$.
Firstly, $(MS)_{R}$ is semisimple since $M_{R}$ is semisimple. Hence, $Y_{R}$
is a direct summand of $(MS)_{R}$. Moreover, $\left\vert G\right\vert
^{-1}\in End_{R}(MS)$ since $G$ is finite and $\left\vert G\right\vert
^{-1}\in End_{R}(M)$. So, $Y_{RG}$ is a direct summand of $(MS)_{RG}$ by
Lemma\ref{semlem2} that means $(MS)_{RG}$ is semisimple.

Assume that $MS$ is a semisimple module over $RG$. $\triangle _{G}(MS)$ is
an $RG$--submodule of $MS$ and we know that $\triangle _{G}(MS)\neq MS$. So, 
$\triangle _{G}(MS)$ is a proper direct summand of $(MS)_{RG}$. Hence, $G$
is a finite group by Lemma \ref{semlem1}.

Let $N$ be an $R$--submodule of $M$. Then, $(NS)_{RG}$ is an $RG$--submodule
of $(MS)_{RG}$. $(NS)_{RG}$ is a direct summand of $(MS)_{RG}$ because $%
(MS)_{RG}$ is semisimple, so there is $\alpha ^{2}=\alpha \in End_{RG}(MS)$
such that $NS=\alpha (MS)$. Let $\alpha \mid _{M}$ be the restriction of $%
\alpha $. Consider the composition such that $\gamma :M\overset{\alpha \mid
_{M}}{\longrightarrow }MS\overset{\varepsilon _{MS}}{\longrightarrow }M$,
and so $\gamma \in End_{R}(M)$. It is clear that $\gamma (M)\subseteq N$.
For any $z\in N$, write $z=\alpha (y)$ where $y\in MG$. Then $\gamma
(z)=\varepsilon _{MS}\alpha (\alpha (y))=\varepsilon _{MS}\alpha
(y)=\varepsilon _{MS}(z)=z$. Hence, $N=\gamma (M)$, $\gamma (\gamma
(z))=\gamma (z)=z$ and $\gamma ^{2}=\gamma $ which means that $N$ is a
direct summand of $M$. Therefore, $M_{R}$ is semisimple $R$--module.

Assume that $\left\vert G\right\vert ^{-1}\notin End_{R}(M)$. Then there is
a prime divisor $p$ of $\left\vert G\right\vert $ such that $p^{-1}$ $\notin
End_{R}(M)$. We prove that $p:M\longrightarrow M$ is not one-to-one. Indeed,
if $p:M\longrightarrow M$ is one-to-one, then $pM\neq M$ because $p^{-1}$ $%
\notin End_{R}(M)$. $M=pM\oplus Z$ for some nonzero $R$--submodule $Z$ of $M$
because $M_{R}$ is semisimple. Since $pM\cap Z=0$, we get $pZ=0$. Thus, $%
p:M\longrightarrow M$ is not one-to-one. So, there exists a nonzero direct
summand $N$ of $M_{R}$ such that $pN=0$ because $M_{R}$ is semisimple.

Now consider $N\hat{G}$ that is an $RG$--submodule of $(MS)_{RG}$ and $N\hat{%
G}\subseteq \triangle _{G}(NS)$ since $\left\vert G\right\vert N=0$. We
claim that $\triangle _{G}(NS)$ is an essential $\ RG$--submodule of $%
(NS)_{RG}$. Let $\sum\limits_{s\in S}n_{s}s\in NS\backslash \triangle
_{G}(NS)$. Then, $0\neq \sum\limits_{s\in S}n_{s}\in N$, and thus $%
(\sum\limits_{s\in S}n_{s}s)\hat{G}=(\sum\limits_{s\in S}n_{s})\hat{G}$ is a
nonzero element of $\triangle _{G}(NS)$. So $\triangle _{G}(NS)$ is an
essential $RG$--submodule of $(NS)_{RG}$. Since $MS$ is a semisimple module
over $RG$ by hypothesis and $(NS)_{RG}$ is submodule of $(MS)_{RG}$, $%
(NS)_{RG}$ is semisimple $RG$--module. Hence, $NS=\triangle _{G}(NS)$, and
so $0=\varepsilon _{MS}(\triangle _{G}(NS))=\varepsilon _{MS}(NS)=N$. This
is a contradiction. So, $\left\vert G\right\vert ^{-1}\in End_{R}(MS)$.
\end{proof}

\textbf{Acknowledgement}

The second author was supported by the Scientific Research Project
Administration of Akdeniz University.

\end{document}